\documentclass{amsart}
\usepackage{latexsym,amssymb,amsmath,amsthm,amscd,graphicx}
\usepackage{esint} 
\usepackage{color}
\numberwithin{equation}{section}
\newtheorem{theorem}{Theorem}[section]
\newtheorem{lemma}[theorem]{Lemma}

\newtheorem{proposition}[theorem]{Proposition}

\theoremstyle{definition}

\newtheorem{definition}[theorem]{Definition}

\theoremstyle{remark}

\begin{document}

\title{A note on the bilinear fractional integral operator acting on Morrey spaces}
\author{Naoya Hatano and Yoshihiro Sawano}
\address[Naoya Hatano]{Department of Mathematics, Chuo University, 1-13-27, Kasuga, Bunkyo-ku, Tokyo 112-8551, Japan}
\address[Yoshihiro Sawano]{Department of Mathematical Science, Tokyo Metropolitan University, 1-1 Minami-Ohsawa, Hachioji, Tokyo, 192-0397, Japan}
\email[Naoya Hatano]{n18012@gug.math.chuo-u.ac.jp}
\email[Yoshihiro Sawano]{ysawano@tmu.ac.jp}
\maketitle

\begin{abstract}
The boundedness of the bilinear fractional integral operator
is investigated.
This bilinear fractional integral operator goes back
to Kenig and Stein.
This paper is oriented to the boundedness
of this operator on products of Morrey spaces.
Compared to the earlier work by He and Yan,
the local integrability condition of the domain
is expanded.
The local integrability condition
can be relaxed with the help of the averaging technique.
\end{abstract}

{\bf Keywords}
Morrey spaces,
bilinear fractional integral operators.

{\bf Mathematics Subject Classifications (2010)} 
Primary 42B35; Secondary 42B25

\section{Introduction}

Let $1<q \le p<\infty$.
Define the {\it Morrey norm}
$\|\cdot\|_{{\mathcal M}^p_q}$ by
\[
\| f \|_{{\mathcal M}^p_q}
\equiv
\sup\left\{
|Q|^{\frac{1}{p}-\frac{1}{q}}\|f\|_{L^q(Q)}
\,:\,\mbox{ $Q$ is a dyadic cube in ${\mathbb R}^n$}\right\}
\]
for a measurable function $f$.
We recall the definition of the dyadic cubes
precisely in Section \ref{s2}.
Here let us content ourselves
with the intuitive understanding that
$p$ serves as the global integrability,
as is hinted by the dilation mapping
$f \mapsto f(t\cdot)$,
and that
$q$ serves as the local integrability.
The {\it Morrey space}
${\mathcal M}^p_q({\mathbb R}^n)$
is the set of all the measurable functions $f$
for which
$\| f \|_{{\mathcal M}^p_q}$
is finite.
A simple geometric observation shows that
\[
\| f \|_{{\mathcal M}^p_q}
\sim
\sup\left\{
|Q|^{\frac{1}{p}-\frac{1}{q}}\|f\|_{L^q(Q)}
\,:\,\mbox{ $Q$ is a cube in ${\mathbb R}^n$}\right\}
\]
for any measurable function $f$.
We handle the following bilinear operator
defined in \cite{Grafakos92, KeSt99}.
\begin{definition}\label{defi:150824-27a}
\index{bilinear fractional integral operator of Grafakos type@bilinear fractional integral operator of Grafakos type}
\index{${\mathcal I}_{\alpha}$}
The {\it bilinear fractional integral operator of Grafakos type}
${\mathcal J}_{\alpha}$, $0<\alpha<n$ is given by
$$
{\mathcal J}_{\alpha}[f_1,f_2](x)
\equiv
\int_{{\mathbb R}^n}
\frac{f_1(x+y)f_2(x-y)}{|y|^{n-\alpha}}{\rm d}y
\quad (x \in {\mathbb R}^n),
$$
where
$f_1,f_2$
are non-negative integrable functions
defined in ${\mathbb R}^n$.
\end{definition}
The operator 
${\mathcal I}_{\alpha}[f_1,f_2]$,
 $0<\alpha<2n$,
defined by
\[
{\mathcal I}_{\alpha}[f_1,f_2](x)
\equiv
\int_{{\mathbb R}^n \times {\mathbb R}^n}
\frac{f_1(y_1)f_2(y_2)}{(|x-y_1|+|x-y_2|)^{2n-\alpha}}{\rm d}y
\quad (x \in {\mathbb R}^n)
\]
for non-negative integrable functions
$f_1$ and $f_2$
defined in ${\mathbb R}^n$,
is a contrast to 
${\mathcal J}_{\alpha}[f_1,f_2]$.
These two operators
with $0<\alpha<n$
pass the fractional integral operator
$I_\alpha$ to the bilinear case,
where
$I_\alpha$ is the fractional maximal operator
\[
I_\alpha f(x)
\equiv
\int_{{\mathbb R}^n}\frac{f(y)}{|x-y|^{n-\alpha}}{\rm d}y
\quad (x \in {\mathbb R}^n)
\]
for a nonnegative measurable function $f:{\mathbb R}^n \to [0,\infty]$.

Here and below
we assume that the functions are non-negative
to ignore the issue of the convergence of the integral
definining
${\mathcal J}_{\alpha}[f_1,f_2](x)$.

The operator
${\mathcal I}_{\alpha}[f_1,f_2]$
acting on Morrey spaces
is investigated by many authors
in many settings such as
the generalized Morrey spaces \cite{BuLi11},
the weighted setting \cite{GuOm15,Iida17}.
the case equipped with the rough kernel \cite{Iida14-2,SSTS16}
and
the non-doubling setting \cite{ISST12,TaZh11-1}.
See also
\cite{DiMe15,XiLu18} for the case of commutators
generated by ${\mathcal I}_\alpha$ and other functions.
However
we do not so much about
the action of
the operator
${\mathcal J}_{\alpha}$
on Morrey spaces.
The works \cite{FaGa12,HeYa19}  considered the boundedness property of
${\mathcal J}_{\alpha}$.
We aim here to prove the following estimate:
\begin{theorem}\label{thm:190204-18}
Let
\[
0<\alpha<n,\quad
1<q_1 \le p_1<\infty,\quad
1<q_2 \le p_2<\infty,\quad
1\le t \le s<\infty.
\]
Define $p$ and $q$ by
\[
\frac{1}{p}=\frac{1}{p_1}+\frac{1}{p_2}, \quad
\frac{1}{q}=\frac{1}{q_1}+\frac{1}{q_2}, 
\]
Assume that 
\[
\frac{1}{s}=\frac{1}{p}-\frac{\alpha}{n}, \quad
\frac{q}{p}=\frac{t}{s}, \quad
s<\min(q_1,q_2).
\]
Then
for all
$f_1 \in {\mathcal M}^{p_1}_{q_1}({\mathbb R}^n)$
and
$f_2 \in {\mathcal M}^{p_2}_{q_2}({\mathbb R}^n)$,
\[
\|{\mathcal J}_\alpha[f_1,f_2]\|_{{\mathcal M}^s_t}
\lesssim
\|f_1\|_{{\mathcal M}^{p_1}_{q_1}}
\|f_2\|_{{\mathcal M}^{p_2}_{q_2}}.
\]
\end{theorem}
As is pointed out in \cite{HeYa19},
the assumption $\frac{q}{p}=\frac{t}{s}$
is essential.
The case $t=1$ is new
and
the case $t>1$ somehow extends the 
work \cite{HeYa19}.

Theorem \ref{thm:190204-18} partially
extends the following result
by Kenig and Stein \cite[Theorem 2]{KeSt99}:
\begin{proposition}
Let
$0<\alpha<n$
and
$1<p_1,p_2<\infty$.
Assume that
$\dfrac{1}{p_1}+\dfrac{1}{p_2}>\dfrac{\alpha}{n}$,
so that we can define $s>0$ by
$\dfrac{1}{s}=\dfrac{1}{p_1}+\dfrac{1}{p_2}-\dfrac{\alpha}{n}$.
Then
for all
$f_1 \in L^{p_1}({\mathbb R}^n)$
and
$f_2 \in L^{p_2}({\mathbb R}^n)$,
\[
\|{\mathcal J}_\alpha[f_1,f_2]\|_{L^s}
\lesssim
\|f_1\|_{L^{p_1}}
\|f_2\|_{L^{p_2}}.
\]
\end{proposition}
In \cite{HeYa19}
He and Yan proved the boundedness
of the operator is to use the H\"{o}lder inequality
under the assumption
\begin{equation}\label{eq:190228-1}
\frac{q_1}{p_1}=\frac{q_2}{p_2}, \quad
\frac{1}{\max\left(q_1',\frac{\alpha}{n}p_1\right)}
+
\frac{1}{\max\left(q_2',\frac{\alpha}{n}p_2\right)}>1,
\end{equation}
so that
\[
\frac{p}{q}=\frac{q_1}{p_1}=\frac{q_2}{p_2}
\]
and there exists $u \in (1,\infty)$ such that
\[
\frac{\alpha}{n}p_1<u<
\left(\frac{\alpha}{n}p_2\right)',
\quad
(q_2)'<u<q_1.
\]
Define 
$s_1,s_2,t_1,t_2$ by
\[
\frac{u}{s_1}=\frac{u}{p_1}-\frac{\alpha}{n}, \quad
\frac{u'}{s_2}=\frac{u'}{p_2}-\frac{\alpha}{n}, \quad
\frac{t_1}{s_1}=\frac{q_1}{p_1}, \quad
\frac{t_2}{s_2}=\frac{q_2}{p_2},
\]
so that
$1<t_1 \le s_1<\infty$
and that
$1<t_2 \le s_2<\infty$.
Then
\[
\frac{1}{s}=\frac{1}{s_1}+\frac{1}{s_2}, \quad
\frac{1}{t}=\frac{1}{t_1}+\frac{1}{t_2},
\]
since
\[
\frac{p}{q}=\frac{q_1}{p_1}=\frac{q_2}{p_2}.
\]
Meanwhile by the H\"{o}lder inequality
we have
\begin{align*}
{\mathcal J}_\alpha[f_1,f_2](x)
\le
\left(
\int_{{\mathbb R}^n}\frac{|f_1(x+y)|^u}{|y|^{\alpha}}{\rm d}y
\right)^{\frac1u}
\left(
\int_{{\mathbb R}^n}\frac{|f_2(x-y)|^{u'}}{|y|^{\alpha}}{\rm d}y
\right)^{\frac1{u'}}
\end{align*}
for any $1<u<\infty$.
Consequently,
by the H\"{o}lder inequality once again,
we obtain
\[
\|{\mathcal J}_\alpha[f_1,f_2]\|_{{\mathcal M}^s_t}
\le
\|I_\alpha^{(u)}f_1\|_{{\mathcal M}^{s_1}_{t_1}}
\|I_\alpha^{(u')}f_2\|_{{\mathcal M}^{s_2}_{t_2}}.
\]
If we use the Adams theorem,
asserting that
$I_\alpha^{(v)}$ maps
${\mathcal M}^P_Q({\mathbb R}^n)$
to
${\mathcal M}^S_T({\mathbb R}^n)$
whenever
$v<Q \le P<\infty$,
$v<T \le S<\infty$,
$\frac{v}{S}=\frac{v}{P}-\frac{\alpha}{n}$
and
$\frac{P}{Q}=\frac{S}{T}$,
we obtain
\[
\|{\mathcal J}_\alpha[f_1,f_2]\|_{{\mathcal M}^s_t}
\lesssim
\|f_1\|_{{\mathcal M}^{p_1}_{q_1}}
\|f_2\|_{{\mathcal M}^{p_2}_{q_2}}.
\]
Thus Theorem \ref{thm:190204-18}
is significant when
(\ref{eq:190228-1}) fails.
See \cite[Theorem 2.2]{FaGa12}
for the bilinear fractional integral operator
of Kenig--Stein type equipped with the rough kernel.

The operator ${\mathcal J}_\alpha$ has a lot to do with the bilinear Hilbert transform
defined by
\[
{\mathcal H}[f_1,f_2](x)
=
\lim_{\varepsilon \downarrow 0}
\int_{{\mathbb R}\setminus(-\varepsilon,\varepsilon)}
\frac{f_1(x+y)f_2(x-y)}{y}{\rm d}y
\quad (x \in {\mathbb R}),
\]
where $f_1$ and $f_2$ are locally integrable functions.
One of the important problems
in harmonic analysis is to investigate 
the boundedness property of the bilinear Hilbert transform.
A conjecture of Calder\'{o}n in 1964 concerned
possible extensions
of ${\mathcal H}$
to a bounded bilinear operator
on products of Lebesgue spaces.
A remarkable fact is that
${\mathcal H}$
maps
$L^{p_1}({\mathbb R}) \times L^{p_2}({\mathbb R})$
to
$L^p({\mathbb R})$
boundedly 
if
$1<p_1\le\infty$,
$1<p_2\le\infty$,
$\dfrac23<p<\infty$
and
$\dfrac1p=\dfrac1{p_1}+\dfrac1{p_2}$
\cite{Lacey1,Lacey2}.
To understand the boundedness property
of this operator,
we consider its counterpart
to fractional integral operators.

\section{Preliminaries}
\label{s2}

For a measurable function $f$
defined on ${\mathbb R}^n$,
define a function $M f$ by
\begin{equation}\label{eq:maximal operator}
M f(x)\equiv
\sup_{B \in {\mathcal B}}\frac{\chi_B(x)}{|B|}\int_B |f(y)|{\rm d}y
\quad (x \in {\mathbb R}^n).
\end{equation}
The mapping $M:f \mapsto Mf$
is called the {\it Hardy--Littlewood maximal operator}.
It is known that the Hardy--Littlewood maximal operator
is bounded on ${\mathcal M}^p_q({\mathbb R}^n)$
if $1<q \le p<\infty$.
A dyadic cube is a set
of the form $Q_{j k}$
for some
$j \in {\mathbb Z}, k=(k_1,k_2,\ldots,k_n) \in {\mathbb Z}^n$.
The set of all dyadic cubes
is denoted by ${\mathcal D}${\rm;}
$\displaystyle
{\mathcal D}
\equiv
\left\{Q_{j k} \, : \, j \in {\mathbb Z}, k \in {\mathbb Z}^n \right\}.
$
For $j\in {\mathbb Z}$
the set of dyadic cubes of the $j$-th generation is given by
\[
{\mathcal D}_j={\mathcal D}_j({\mathbb R}^n)\equiv \{Q_{j k} \, : \, k \in {\mathbb Z}^n \}
=\{Q \in {\mathcal D} \, : \, \ell(Q)=2^{-j} \}.
\]

The following lemma can be located
as a standard estimate to handle
this bilinear fractional integral operator.
\begin{lemma}\label{lem:180615-41}
Let $f_1,f_2 \ge 0$ be measurable functions.
Then we have
\begin{align*}
{\mathcal J}_{\alpha}[f_1,f_2](x)
&\lesssim
\sum_{l=-\infty}^{\infty}
\sum_{Q \in {\mathcal D}_l}
2^{l(n-\alpha)}\chi_Q(x)
\int_{B(2^{-l})}
f_1(x+y)f_2(x-y){\rm d}y
\quad (x \in {\mathbb R}^n).
\end{align*}
\end{lemma}

\begin{proof}
We will follow the idea
used in \cite[Theorem 2]{KeSt99}.
See also \cite[Theorem 3.2]{Moen16}
and \cite{Komori19-1,Komori19-2} as well.
We decompose
\begin{align*}
{\mathcal J}_{\alpha}[f_1,f_2](x)
&=
\int_{{\mathbb R}^n}
\frac{f_1(x+y)f_2(x-y)}{|y|^{n-\alpha}}{\rm d}y
\\
&=
\sum_{l=-\infty}^{\infty}
\int_{B(2^{-l}) \setminus B(2^{-l-1})}
\frac{f_1(x+y)f_2(x-y)}{|y|^{n-\alpha}}{\rm d}y\\
&\sim
\sum_{l=-\infty}^{\infty}
2^{l(n-\alpha)}
\int_{B(2^{-l}) \setminus B(2^{-l-1})}
f_1(x+y)f_2(x-y){\rm d}y\\
&\le
\sum_{l=-\infty}^{\infty}
2^{l(n-\alpha)}
\int_{B(2^{-l})}
f_1(x+y)f_2(x-y){\rm d}y.
\end{align*}
Observe that for each $l \in {\mathbb N}$
there uniquely exists a dyadic cube $Q \in {\mathcal D}_l$
such that
$x \in Q$.
Thus, we obtain the desired result.
\end{proof}

\begin{lemma}\label{thm:131108-2}
Suppose that the parameters $p,q,s,t$ satisfy
$$
1<q \le p<\infty, \quad 1<t \le s<\infty, \quad
q<t, \quad p<s
$$
or
$$
1=q \le p<\infty, \quad 1=t \le s<\infty, \quad p<s.
$$
Assume that
$\{Q_j\}_{j=1}^{\infty} \subset {\mathcal D}({\mathbb R}^n)$,
$\{a_j\}_{j=1}^{\infty} \subset {\mathcal M}^s_t({\mathbb R}^n)$
and
$\{\lambda_j\}_{j=1}^{\infty} \subset[0,\infty)$
fulfill
\begin{equation}\label{eq:thm1-1}
{\rm supp}(a_j) \subset Q_j, \quad
\left\|\sum_{j=1}^{\infty} \lambda_j \chi_{Q_j}\right\|_{{\mathcal M}^p_q}
<\infty.
\end{equation}
Then
$\displaystyle f=\sum_{j=1}^{\infty} \lambda_j a_j$
converges in
${\mathcal S}'({\mathbb R}^n) \cap L^q_{\rm loc}({\mathbb R}^n)$
and satisfies
\begin{equation}\label{eq:thm1-2}
\|f\|_{{\mathcal M}^p_q}
\lesssim_{p,q,s,t}
\left\|\sum_{j=1}^{\infty}\lambda_j\frac{\|a_j\|_{{\mathcal M}^s_t}}{|Q_j|^{\frac1s}}\chi_{Q_j}\right\|_{{\mathcal M}^p_q}.
\end{equation}
\end{lemma}

\begin{proof}
This estimate is essentially obtained in \cite{IST14} if $q>1$ and \cite{GHSN16} if $q=1$.
Although we distinguished these cases in these papers,
we can combine them,
since the case of $q=1$ can almost be emerged into the case of $q>1$.

Let us suppose $q>1$ for the time being.
Let $0<\eta<\infty$.
We will use the {\it powered Hardy--Littlewood maximal operator}
$M^{(\eta)}$
defined by
\[
M^{(\eta)}f(x)
\equiv \sup_{R>0}
\left(
\frac{1}{|B(x,R)|}\int_{B(x,R)}|f(y)|^\eta{\rm d}y
\right)^\frac{1}{\eta}
\]
for a measurable function $f:{\mathbb R}^n \to {\mathbb C}$.
If $\eta=1$, then we write $M$ instead of $M^{(\eta)}$.
To prove this, we resort to the duality. For the time being, we
assume that there exists $N \in {\mathbb N}$ such that
$\lambda_j=0$ whenever $j \ge N$. Let us assume in addition that
the $a_j$ are non-negative. Fix a non-negative
function $g$ that is supported on a cube
$Q$ such that
$\|g\|_{L^{q'}} \le |Q|^{\frac1{q'}-\frac1{p'}}$.
We will show
\begin{equation}\label{eq:140812-102}
\int_{{\mathbb R}^n}f(x)g(x){\rm d}x
\lesssim_{p,q,s,t}
\left\|\sum_{j=1}^{\infty}\lambda_j\frac{\|a_j\|_{{\mathcal M}^s_t}}{|Q_j|^{\frac1s}}\chi_{Q_j}\right\|_{{\mathcal M}^p_q}.
\end{equation}
to obtain
\[
\|f\|_{{\mathcal M}^p_q}
\lesssim_{p,q,s,t}
\left\|\sum_{j=1}^{\infty}\lambda_j\frac{\|a_j\|_{{\mathcal M}^s_t}}{|Q_j|^{\frac1s}}\chi_{Q_j}\right\|_{{\mathcal M}^p_q}.
\]

Assume first that each $Q_j$ contains $Q$ as a proper subset.
If we group $j$'s such that $Q_j$ are identical,
we can assume that
each $Q_j$
is a $j$-th parent of $Q$ for each $j \in {\mathbb N}$.
Then we have
\begin{align*}
\int_{{\mathbb R}^n}f(x)g(x){\rm d}x
=
\sum_{j=1}^{\infty} \lambda_j
\int_{Q}a_j(x)g(x){\rm d}x
\le
\sum_{j=1}^{\infty} \lambda_j
\|a_j\|_{L^q(Q)}\|g\|_{L^{q'}(Q)}
\end{align*}
from $f=\sum_{j=1}^{\infty} \lambda_j a_j$.
By the size condition of $a_j$ and $g$,
we obtain
\begin{align*}
\int_{{\mathbb R}^n}f(x)g(x){\rm d}x
\le
\sum_{j=1}^{\infty} \lambda_j\|a_j\|_{{\mathcal M}^s_t}
|Q|^{\frac{1}{q}-\frac{1}{s}}|Q|^{\frac{1}{q'}-\frac{1}{p'}}
=
\sum_{j=1}^{\infty} \lambda_j\|a_j\|_{{\mathcal M}^s_t}
|Q|^{\frac{1}{p}-\frac{1}{s}}.
\end{align*}
Note that
\[
\left\|\sum_{j=1}^{\infty} \lambda_j\frac{\|a_j\|_{{\mathcal M}^s_t}}{|Q_j|^{\frac1s}} \chi_{Q_j}\right\|_{{\mathcal M}^p_q}
\ge
\frac{\|a_{j_0}\|_{{\mathcal M}^s_t}}{|Q_{j_0}|^{\frac1s}}
\left\|\lambda_{j_0} \chi_{Q_{j_0}}\right\|_{{\mathcal M}^p_q}
=\|a_{j_0}\|_{{\mathcal M}^s_t}
|Q_{j_0}|^{\frac{1}{p}-\frac{1}{s}}\lambda_{j_0}
\]
for each $j_0$.
Consequently, it follows from the condition $p<s$ that
\begin{align*}
\int_{{\mathbb R}^n}f(x)g(x){\rm d}x
\le
\sum_{j=1}^{\infty}
|Q|^{\frac{1}{p}-\frac{1}{s}}|Q_j|^{\frac{1}{s}-\frac{1}{p}}
\cdot
\left\|\sum_{k=1}^{\infty} \lambda_k\frac{\|a_k\|_{{\mathcal M}^s_t}}{|Q_k|^{\frac1s}} \chi_{Q_k}\right\|_{{\mathcal M}^p_q}
\lesssim
\left\|\sum_{j=1}^{\infty} \lambda_j\frac{\|a_j\|_{{\mathcal M}^s_t}}{|Q_j|^{\frac1s}} \chi_{Q_j}\right\|_{{\mathcal M}^p_q}.
\end{align*}
Conversely, assume that $Q$ contains each $Q_j$.
Then we have
\begin{align*}
\int_{{\mathbb R}^n}f(x)g(x){\rm d}x
=
\sum_{j=1}^{\infty} \lambda_j
\int_{Q_j}a_j(x)g(x){\rm d}x
\le
\sum_{j=1}^{\infty} \lambda_j
\|a_j\|_{L^t(Q_j)}\|g\|_{L^{t'}(Q_j)}.
\end{align*}
By the condition of $a_j$,
we obtain
\[
\int_{{\mathbb R}^n}f(x)g(x){\rm d}x
=
\sum_{j=1}^{\infty} \lambda_j
\int_{Q_j}a_j(x)g(x){\rm d}x
\le
\sum_{j=1}^{\infty} \lambda_j\|a_j\|_{{\mathcal M}^s_t}
|Q_j|^{\frac1t-\frac1s}\|g\|_{L^{t'}(Q_j)}.
\]
Thus, in terms of the Hardy--Littlewood maximal operator $M$,
we obtain
\begin{align*}
\int_{{\mathbb R}^n}f(x)g(x){\rm d}x
&\le
\sum_{j=1}^{\infty}
\lambda_j\frac{\|a_j\|_{{\mathcal M}^s_t}}{|Q_j|^{\frac1s}}  |Q_j|\times \inf_{y \in Q_j}M^{(t')}g(y)\\
&\le
\int_{{\mathbb R}^n}
\left(\sum_{j=1}^{\infty} \lambda_j\frac{\|a_j\|_{{\mathcal M}^s_t}}{|Q_j|^{\frac1s}}  \chi_{Q_j}(y)\right)
M^{(t')}g(y){\rm d}y\\
&\le
\int_{{\mathbb R}^n}
\left(\sum_{j=1}^{\infty} \lambda_j\frac{\|a_j\|_{{\mathcal M}^s_t}}{|Q_j|^{\frac1s}}  \chi_{Q_j}(y)\right)
\chi_{Q}(y)M^{(t')}g(y){\rm d}y.
\end{align*}
Hence, we obtain
(\ref{eq:140812-102})
by the H\"{o}lder inequality,
since
$\|\chi_Q M^{(t')}g\|_{L^{q'}} \lesssim
|Q|^{\frac1p-\frac1q}$.
Thus the proof for the case of $q>1$ is complete.

The case of $q=1$ is a minor modification of the above proof.
First,
if each $Q_j$ contains $Q$ as a proper subset,
the same argument as above works.
If each $Q$ contains $Q_j$,
then we can take $g=|Q|^{\frac1p-1}\chi_Q$ 
to obtain
\[
|Q|^{\frac1p-1}\int_Q f(x)g(x){\rm d}x
\lesssim_{p,s}
\left\|\sum_{j=1}^{\infty}\lambda_j\frac{\|a_j\|_{{\mathcal M}^s_1}}{|Q_j|^{\frac1s}}\chi_{Q_j}\right\|_{{\mathcal M}^p_1}.
\]
We go through the same argument
as before, where
we will replace $M^{(t')}g$ by $1$.
Since
$\|\chi_Q 1\|_{L^\infty} \lesssim
|Q|^{\frac1p-1}$,
we do not have to resort to the boundedness
of the maximal operator $M^{(t')}$ as we did
in the estimate
$\|\chi_Q M^{(t')}g\|_{L^\infty} \lesssim
|Q|^{\frac1p-1}$.
So the proof is complete in this case.
\end{proof}

\begin{lemma}\label{thm:180615-5}
Let
\[
0<\alpha<2 n, \quad
1<q_j \le p_j<\infty, \quad
0<q \le p<\infty, \quad
0<t \le s<\infty
\]
for $j=1,2$.
Assume
\[
\frac{1}{p}=\frac{1}{p_1}+\frac{1}{p_2}, \quad
\frac{1}{q}=\frac{1}{q_1}+\frac{1}{q_2}, 
\]
\[
\frac{1}{s}=\frac{1}{p}-\frac{\alpha}{n}, \quad
\frac{q}{p}=\frac{t}{s}.
\]
Then
\begin{equation}\label{eq:190213-1}
|R|^{\frac1s-\frac1t}
\left\|
\sum_{Q \in {\mathcal D}}
\frac{\chi_Q}{\ell(Q)^{2n-\alpha}}
\int_{(3Q)^2}f_1(y_1)f_2(y_2){\rm d}y_1{\rm d}y_2
\right\|_{L^t(R)}
\lesssim
\prod_{j=1}^2 \|f_j\|_{{\mathcal M}^{p_j}_{q_j}}
\end{equation}
for any cube $R$
and
for all non-negative measurable functions 
$f_1,f_2$.
\end{lemma}

See the proof of \cite[Theorem 2.2]{FaGa12}
for a similar approach.
\begin{proof}
Let $L=L(x)$ be a  positive number that is specified shortly. We decompose
\begin{align*}
&\sum_{Q \in {\mathcal D}}\frac{\chi_Q(x)}{\ell(Q)^{2n-\alpha}}\int_{(3Q)^2}f_1(y_1)f_2(y_2){\rm d}y_1{\rm d}y_2\\
&\hspace{2cm}=\sum_{Q \in {\mathcal D},l(Q)\le L}\frac{\chi_Q(x)}{\ell(Q)^{2n-\alpha}}\int_{(3Q)^2}f_1(y_1)f_2(y_2){\rm d}y_1{\rm d}y_2\\
&\hspace{2cm}\quad+\sum_{Q \in {\mathcal D},l(Q)>L}\frac{\chi_Q(x)}{\ell(Q)^{2n-\alpha}}\int_{(3Q)^2}f_1(y_1)f_2(y_2){\rm d}y_1{\rm d}y_2\\
&\hspace{2cm}=:S_1+S_2.
\end{align*}
First, we estimate the quantity $S_1$.
\[
S_1\lesssim\sum_{Q\in{\mathcal D},l(Q)\le L}\chi_Q(x)l(Q)^\alpha Mf_1(x)Mf_2(x)\sim L^\alpha Mf_1(x)Mf_2(x).
\]
Next, we estimate the quantity $S_2$. By H$\ddot{\rm o}$lder's inequality,
\begin{align*}
S_2&\lesssim\sum_{Q \in {\mathcal D},l(Q)>L}\frac{\chi_Q(x)}{\ell(Q)^{2n-\alpha}}|Q|^{\frac1{q_1'}}\|f_1\|_{L^{q_1}(3Q)}\cdot|Q|^{\frac1{q_2'}}\|f_1\|_{L^{q_2}(3Q)}\\
&\lesssim\sum_{Q \in {\mathcal D},l(Q)>L}\chi_Q(x)|Q|^{-\frac1s}\|f_1\|_{{\mathcal M}^{p_1}_{q_1}}\|f_2\|_{{\mathcal M}^{p_2}_{q_2}}\\
&\sim L^{-\frac ns}\|f_1\|_{{\mathcal M}^{p_1}_{q_1}}\|f_2\|_{{\mathcal M}^{p_2}_{q_2}}.
\end{align*}
Hence we obtain
\[
\sum_{Q \in {\mathcal D}}\frac{\chi_Q(x)}{\ell(Q)^{2n-\alpha}}\int_{(3Q)^2}f_1(y_1)f_2(y_2){\rm d}y_1{\rm d}y_2\lesssim L^\alpha Mf_1(x)Mf_2(x)+L^{-\frac ns}\|f_1\|_{{\mathcal M}^{p_1}_{q_1}}\|f_2\|_{{\mathcal M}^{p_2}_{q_2}}.
\]
In particular, choose the constant $L=L(x)$ 
to optimize the right-hand side:
\[
L=\left(\frac{\|f_1\|_{{\mathcal M}^{p_1}_{q_1}}\|f_2\|_{{\mathcal M}^{p_2}_{q_2}}}{Mf_1(x)Mf_2(x)}\right)^{\frac pn}.
\]
Then we have
\[
\sum_{Q \in {\mathcal D}}\frac{\chi_Q(x)}{\ell(Q)^{2n-\alpha}}\int_{(3Q)^2}f_1(y_1)f_2(y_2){\rm d}y_1{\rm d}y_2\lesssim(Mf_1(x)Mf_2(x))^{\frac ps}\left(\|f_1\|_{{\mathcal M}^{p_1}_{q_1}}\|f_2\|_{{\mathcal M}^{p_2}_{q_2}}\right)^{1-\frac ps}.
\]
Therefore, using H$\ddot{\rm o}$lder's inequality
for Morrey spaces,
the ${\mathcal M}^{p_1}_{q_1}({\mathbb R}^n)$-boundedness
of $M$
and
the ${\mathcal M}^{p_1}_{q_1}({\mathbb R}^n)$-boundedness
of $M$,
we have
\begin{align*}
&|R|^{\frac1s-\frac1t}\left\|\sum_{Q \in {\mathcal D}}\frac{\chi_Q}{\ell(Q)^{2n-\alpha}}\int_{(3Q)^2}f_1(y_1)f_2(y_2){\rm d}y_1{\rm d}y_2\right\|_{L^t(R)}\\
&\hspace{2cm}\lesssim\left\|(Mf_1\cdot Mf_2)^{\frac ps}\right\|_{{\mathcal M}^s_t}\left(\|f_1\|_{{\mathcal M}^{p_1}_{q_1}}\|f_2\|_{{\mathcal M}^{p_2}_{q_2}}\right)^{1-\frac ps}\\
&\hspace{2cm}=\|Mf_1\cdot Mf_2\|_{{\mathcal M}^p_q}^{\frac ps}\left(\|f_1\|_{{\mathcal M}^{p_1}_{q_1}}\|f_2\|_{{\mathcal M}^{p_2}_{q_2}}\right)^{1-\frac ps}\\
&\hspace{2cm}\le\left(\|Mf_1\|_{{\mathcal M}^{p_1}_{q_1}}\|Mf_2\|_{{\mathcal M}^{p_2}_{q_2}}\right)^{\frac ps}\left(\|f_1\|_{{\mathcal M}^{p_1}_{q_1}}\|f_2\|_{{\mathcal M}^{p_2}_{q_2}}\right)^{1-\frac ps}\\
&\hspace{2cm}\lesssim\|f_1\|_{{\mathcal M}^{p_1}_{q_1}}\|f_2\|_{{\mathcal M}^{p_2}_{q_2}}.
\end{align*}
\end{proof}

\section{Proof of Theorem \ref{thm:190204-18}}

Let $v\in(s,\min(q_1,q_2))$.
Let $x \in Q \in {\mathcal D}_l$.
By the Minkowski inequality and the H\"{o}lder inequality
\begin{align*}
\left\|\int_{B(2^{-l})}
f_1(\cdot+y)f_2(\cdot-y){\rm d}y\right\|_{L^v(Q)}
&\le
\int_{B(2^{-l})}
\left\|f_1(\cdot+y)f_2(\cdot-y)\right\|_{L^v(Q)}
{\rm d}y\\
&\le
|B(2^{-l})|^{\frac{1}{v'}}
\left(\int_{B(2^{-l})}
\left\|f_1(\cdot+y)f_2(\cdot-y)\right\|_{L^v(Q)}{}^v
{\rm d}y\right)^{\frac1v}\\
&\lesssim
|B(2^{-l})|^{\frac{1}{v'}}
\|f_1\|_{L^v(Q(x,3 \cdot 2^{-l}))}
\|f_2\|_{L^v(Q(x,3 \cdot 2^{-l}))}\\
&\lesssim
|B(2^{-l})|^{1+\frac{1}{v}}
\inf_{y_1 \in Q}M^{(v)} f_1(y_1)
\inf_{y_2 \in Q}M^{(v)} f_2(y_2).
\end{align*}
Then
thanks to Theorem \ref{thm:131108-2}
\begin{align*}
\lefteqn{
\|{\mathcal J}_{\alpha}[f_1,f_2]\|_{{\mathcal M}^s_t}
}\\
&\lesssim
\left\|
\sum_{l=-\infty}^{\infty}
\sum_{Q \in {\mathcal D}_l}2^{l(n-\alpha)}
\frac{\chi_Q}{|Q|^{\frac1v}}
\left\|\int_{B(2^{-l})}
f_1(\cdot+y)f_2(\cdot-y){\rm d}y\right\|_{L^v(Q)}
\right\|_{{\mathcal M}^s_t}\\
&\lesssim
\left\|
\sum_{l=-\infty}^{\infty}
\sum_{Q \in {\mathcal D}_l}
2^{-l\alpha}
\frac{\chi_Q}{|Q|}
\int_{3Q}M^{(v)}f_1(y_1){\rm d}y_1
\cdot
\frac{1}{|Q|}
\int_{3Q}M^{(v)}f_2(y_2){\rm d}y_2
\right\|_{{\mathcal M}^s_t}.
\end{align*}
Thus, we are again in the position of using (\ref{eq:190213-1})
to have 
\[
\|{\mathcal J}_{\alpha}[f_1,f_2]\|_{{\mathcal M}^s_t}
\lesssim
\|M^{(v)}f_1\|_{{\mathcal M}^{p_1}_{q_1}}
\|M^{(v)}f_2\|_{{\mathcal M}^{p_2}_{q_2}}.
\]
Since $v<q_1,q_2$,
we are in the position of using the boundedness
of $M$ on Morrey spaces obtained by
Chiarenza and Frasca \cite{ChFr87}.
If we use the boundedness of the
Hardy--Littlewood maximal operator,
then we obtain
\[
\|{\mathcal J}_{\alpha}[f_1,f_2]\|_{{\mathcal M}^s_t}
\lesssim
\|f_1\|_{{\mathcal M}^{p_1}_{q_1}}
\|f_2\|_{{\mathcal M}^{p_2}_{q_2}}.
\]
This is the desired result.

To conclude the paper
we remark that Fan and Gao obtained an
estimate to control
\[
\left\|\int_{B(2^{-l})}
f_1(\cdot+y)f_2(\cdot-y){\rm d}y\right\|_{L^v(Q)}
\]
in \cite[Lemma 2.1]{FaGa12}.

{\bf Acknowledgements.} 
The second-named author was supported by Grant-in-Aid for Scientific
Research (C) (16K05209), the Japan Society for the Promotion of Science.
The authors are thankful to Professor Komori--Furuya Yasuo for his
kind introcduction to \cite{Komori19-1,Komori19-2}.

\end{document}